\documentclass[a4paper,11pt]{amsart}
\usepackage{hyperref}
\usepackage[all]{xy}
\SelectTips{cm}{}

\newcommand{\ko}{\: , \;}
\newcommand{\ul}[1]{\underline{#1}}

\numberwithin{equation}{subsection}
\newtheorem{theorem}{Theorem}
\newtheorem{classification-theorem}[subsection]{Classification Theorem}
\newtheorem{decomposition-theorem}[subsection]{Decomposition Theorem}
\newtheorem{proposition-definition}[subsection]{Proposition-Definition}
\newtheorem{periodicity-conjecture}[subsection]{Periodicity Conjecture}

\newtheorem{corollary}[theorem]{Corollary}

\newcommand{\reminder}[1]{}

\newcommand{\opname}[1]{\operatorname{\mathsf{#1}}}

\renewcommand{\mod}{\opname{mod}\nolimits}

\newcommand{\Mod}{\opname{Mod}\nolimits}

\newcommand{\per}{\opname{per}\nolimits}
\newcommand{\add}{\opname{add}\nolimits}

\newcommand{\Z}{\mathbb{Z}}

\newcommand{\iso}{\xrightarrow{_\sim}}

\newcommand{\Hom}{\opname{Hom}}

\newcommand{\End}{\opname{End}}

\newcommand{\ten}{\otimes}
\newcommand{\lten}{\overset{\boldmath{L}}{\ten}}

\newcommand{\ca}{{\mathcal A}}

\newcommand{\cd}{{\mathcal D}}

\newcommand{\ct}{{\mathcal T}}

\renewcommand{\phi}{\varphi}

\begin{document}

\date{May 31, 2018}

\title{A remark on a theorem by C.~Amiot}
\author{Bernhard Keller}
\address{Universit\'e Paris Diderot -- Paris 7\\
    UFR de Math\'ematiques\\
   Institut de Math\'ematiques de Jussieu--PRG, UMR 7586 du CNRS \\
   Case 7012\\
    B\^{a}timent Chevaleret\\
    75205 Paris Cedex 13\\
    France
}
\email{bernhard.keller@imj-prg.fr}

\begin{abstract}
C.~Amiot has classified the connected triangulated $k$-categories with
finitely many isoclasses of indecomposables satisfying suitable hypotheses.
We remark that her proof shows that these triangulated categories are determined
by their underlying $k$-linear categories. We observe that if the connectedness
assumption is dropped, the triangulated categories are still determined by
their underlying $k$-categories together with the action of the suspension
functor on the set of isoclasses of indecomposables.
\end{abstract}

\keywords{Additively finite triangulated category}

\subjclass[2010]{18E30}


\maketitle

\section{The connected case}

We refer to \cite{Amiot07a} for unexplained notation
and terminology.
Let $k$ be an algebraically closed field and $\ct$ a $k$-linear $\Hom$-finite
triangulated category with split idempotents. Recall Theorem~7.2
of \cite{Amiot07a}:

\begin{theorem}[Amiot] 
Suppose $\ct$ is connected, algebraic, standard and has
only finitely many isoclasses of indecomposables. Then there exists a 
Dynkin quiver $Q$ and a triangle autoequivalence $\Phi$ of $\cd^b(\mod kQ)$ such
that $\ct$ is triangle equivalent to the triangulated \cite{Keller05}
orbit category
$\cd^b(\mod kQ)/\Phi$.
\end{theorem}

Our aim is to show that the proof of this theorem
in \cite{Amiot07a} actually shows that a given $k$-linear
equivalence $\cd^b(\mod kQ)/F \iso \ct$, where $F$ is a $k$-linear
autoequivalence, lifts to a triangle equivalence
$\cd^b(\mod kQ)/\Phi \iso \ct$, where $\Phi$ is a triangle autoequivalence
lifting $F$. Thus, we obtain the

\begin{corollary} Under the hypotheses of the theorem, the $k$-linear structure
of $\ct$ determines its triangulated structure up to triangle equivalence.
\end{corollary}

\begin{proof} The facts that $\ct$ is connected, standard and has only finitely
many isoclasses of indecomposables imply that there is a Dynkin quiver $Q$,
a $k$-linear autoequivalence $F$ of $\cd^b(\mod kQ)$ and a $k$-linear
equivalence
\[
G: \cd^b(\mod kQ)/F \iso \ct.
\]
This follows from the work of Riedtmann \cite{Riedtmann80}, cf. section~6.1
of \cite{Amiot07a}. We will show that $F$ lifts to an (algebraic) triangle
autoequivalence of $\cd^b(\mod kQ)$ and $G$ to an (algebraic) triangle 
equivalence $\Gamma$. We give the details in the case of $F$ which
were omitted in \cite{Amiot07a}. Put $\cd=\cd^b(\mod kQ)$. Since $\cd$ is
triangulated, the category $\mod\cd$ of finitely presented functors
$\cd^{op}\to \Mod k$ is an exact Frobenius category and we have
a canonical isomorphism of functors $\ul{\mod}\,\cd \to \ul{\mod}\,\cd$
\[
\Sigma_m^3 \iso \Sigma_\cd \ko
\]
where $\Sigma_\cd :\mod\cd \to \mod\cd$ denotes the functor
$\mod \cd \to \mod\cd$ induced by $\Sigma$
and $\Sigma_m$ is the suspension functor of the
stable category $\ul{\mod}\,\cd$, cf. \cite[16.4]{Heller68}. Notice that $\Sigma_m$ only
depends on the underlying $k$-category of $\cd$. Now if $S_U$
denotes the simple $\cd$-module associated with an indecomposable object $U$
of $\cd$, we have
\[
S_{\Sigma U} = \Sigma_\cd S_U \iso \Sigma_m^3 S_U
\]
in the stable category $\ul{\mod}\,\cd$.
Since $F$ is a $k$-linear autoequivalence, the functor it induces
in $\ul{\mod}\,\cd$ commutes with $\Sigma_m$ and
we have $S_{F\Sigma U} \cong S_{\Sigma FU}$ in $\ul{\mod}\,\cd$
for each indecomposable $U$ of $\cd$. It follows that if $S_U$ is
not zero in $\ul{\mod}\,\cd$, then we have an isomorphism
$F\Sigma U \cong \Sigma F U$ in $\cd$. Now $S_U$ is zero in
$\ul{\mod}\, \cd$ only if $S_U$ is projective, which happens if and
only if the canonical map $P_U \to S_U$ is an isomorphism, where
$P_U=\cd(?,U)$ is the projective module associated with $U$.
This is the case only if no arrows arrive at $U$ in the Auslander--Reiten
quiver of $\cd$ and this happens if and only if no arrows start or arrive
at $U$. The same then holds for the suspensions $\Sigma^n U$, $n\in\Z$.
Since we have assumed that $\ct$ and hence $\cd$ is connected, this
case is impossible. Therefore, we have an isomorphism $\Sigma F U\cong F \Sigma U$
for each indecomposable $U$ of $\cd$.
It follows that $T=F(kQ)$ is a tilting object of $\cd$. By \cite{Keller00}, we
can lift $T$ to a $kQ$-bimodule complex $Y$ which is even unique in the
derived category of bimodules if we take the isomorphism 
$kQ \iso \End_\cd(F(kQ))$ into account. Since the $k$-linear functors $F$ and 
$\Phi=?\lten_{kQ} Y$ are isomorphic when restricted to $\add(kQ)$, they
are isomorphic as $k$-linear functors by Riedtmann's knitting argument
\cite{Riedtmann80}. Since the triangulated category $\ct$
is algebraic, we may assume that it equals the perfect derived category
$\per \ca$ of a small dg $k$-category $\ca$. Using Riedtmann's knitting
argument again, it follows from the proof of Theorem~7.2
in \cite{Amiot07a} that the composition
\[
\xymatrix{\cd \ar[r]^-{\pi} & \cd/\Phi \ar[r]^-{G} & \ct=\per\ca}
\]
lifts to a triangle functor $?\lten_{kQ} X$ for a $kQ$-$\ca$-bimodule $X$.
Moreover, it is shown there that this composition factors through an
algebraic triangle equivalence
\[
\Gamma: \cd/\Phi \iso \ct=\per\ca.
\]
Since the compositions $\Gamma\circ \pi$ and $G\circ \pi$ are isomorphic
as $k$-linear functors, the functors $\Gamma$ and $G$ are isomorphic as 
$k$-linear functors.
\end{proof}

\section{The non connected case}

Let $k$ be an algebraically closed field and $\ct$ a $k$-linear $\Hom$-finite
triangulated category with split idempotents and finitely many isomorphism
classes of indecomposables. We assume that $\ct$ is algebraic and standard
but possibly non connected. 

Assume first that $\ct$ is $\Sigma$-connected,
i.e. that the $k$-linear orbit category $\ct/\Sigma$ is connected.
Then the argument at the beginning of the above proof shows that
either $\ct$ is connected or $\ct$ is $k$-linearly equivalent to
$\cd^b(\mod kA_1)/F$ for a $k$-linear equivalence $F$ of $\cd^b(\mod kA_1)$.
Clearly $F$ lifts to a triangle autoequivalence, namely a power $\Sigma^N$,
of the suspension functor of $\cd^b(\mod kA_1)$. We may assume $N>0$ equals
the number of isoclasses of indecomposables of $\ct$. Since the underlying
$k$-category of $\ct$ is abelian and semi-simple, all triangles of $\ct$
split and $\ct$ is triangle equivalent to $\cd^b(\mod kA_1)/\Sigma^N$.

Let us now drop the $\Sigma$-connectedness assumption on $\ct$.
Then clearly $\ct$ decomposes, as a triangulated category, into
finitely many $\Sigma$-connected components (the pre-images of
the connected components of $\ct/\Sigma$). Each of these is either
connected or triangle equivalent to $\cd^b(\mod kA_1)/\Sigma^N$
for some $N>0$. Thus, the indecomposables of $\ct$ either lie in
connected components or in $\Sigma$-connected non connected
components and the triangle equivalence class of the latter is determined by the action
of $\Sigma$ on the isomorphism classes of indecomposables.
We obtain the

\begin{corollary} $\ct$ is determined up to triangle equivalence by its
underlying $k$-category and the action of $\Sigma$ on the set of
isomorphism classes of indecomposables.
\end{corollary}

We refer to Theorem~6.5 of \cite{Hanihara18} for an analogous result
concerning the $\Sigma$-finite triangulated categories $\ct$ and to
\cite{Crawford16} for an application.

I am grateful to Susan Sierra for bringing the question to my attention.


\begin{thebibliography}{1}

\bibitem{Amiot07a}
Claire Amiot, \emph{On the structure of triangulated categories with finitely
  many indecomposables}, Bull. Soc. Math. France \textbf{135} (2007), no.~3,
  435--474.

\bibitem{Crawford16}
Simon Crawford, \emph{Singularity categories of deformations of {K}leinian
  singularities}, arXiv:1610.08430 [math.RA].

\bibitem{Hanihara18}
Norihiro Hanihara, \emph{Auslander correspondence for triangulated categories},
  arXiv:1805.07585 [math.RT].

\bibitem{Heller68}
Alex Heller, \emph{Stable homotopy categories}, Bull. Amer. Math. Soc.
  \textbf{74} (1968), 28--63.

\bibitem{Keller00}
Bernhard Keller, \emph{Bimodule complexes via strong homotopy actions}, Algebr.
  Represent. Theory \textbf{3} (2000), no.~4, 357--376, Special issue dedicated
  to Klaus Roggenkamp on the occasion of his 60th birthday.

\bibitem{Keller05}
\bysame, \emph{{On triangulated orbit categories}}, Doc. Math. \textbf{10}
  (2005), 551--581.

\bibitem{Riedtmann80}
Christine Riedtmann, \emph{Algebren, {D}arstellungsk{\"o}cher,
  \"{U}berlagerungen und zur{\"u}ck}, Comment. Math. Helv. \textbf{55} (1980),
  no.~2, 199--224.

\end{thebibliography}

\def\cprime{$'$} \def\cprime{$'$}
\providecommand{\bysame}{\leavevmode\hbox to3em{\hrulefill}\thinspace}
\providecommand{\MR}{\relax\ifhmode\unskip\space\fi MR }
\providecommand{\MRhref}[2]{%
  \href{http://www.ams.org/mathscinet-getitem?mr=#1}{#2}
}
\providecommand{\href}[2]{#2}

\end{document}